\documentclass[12pt]{article}
\RequirePackage{etex}

\usepackage[margin=1in]{geometry}
\usepackage[english]{babel}
\usepackage[utf8]{inputenc}
\usepackage{subcaption}
\usepackage{mathtools}
\usepackage{amssymb}
\usepackage{amsfonts}
\usepackage{amsthm}
\usepackage{mathrsfs}
\usepackage{epstopdf}
\usepackage{inputenc}
\usepackage[all]{xy}
\usepackage[pdftex]{graphicx}
\usepackage{color}
\usepackage{cite}
\usepackage{url}
\usepackage{indent first}
\usepackage[labelfont=bf,labelsep=period,justification=raggedright]{caption}
\usepackage[english]{babel}
\usepackage[utf8]{inputenc}
\usepackage{hyperref}
\usepackage[colorinlistoftodos]{todonotes}
\usepackage{tkz-fct}
\usepackage{tikz}
\usetikzlibrary{calc}
\usepackage{multicol}
\PassOptionsToPackage{dvipsnames,svgnames}{xcolor}
\usepackage{textcomp}

\def\multiset#1#2{\ensuremath{\left(\kern-.3em\left(\genfrac{}{}{0pt}{}{#1}{#2}\right)\kern-.3em\right)}}

\theoremstyle{plain}
\newtheorem{theorem}{Theorem}

\newtheorem{proposition}[theorem]{Proposition}

\newtheorem{remark}[theorem]{Remark}

\numberwithin{equation}{section}
\numberwithin{theorem}{section}

\usepackage{arxiv}

\usepackage[utf8]{inputenc} 
\usepackage[T1]{fontenc}    
\usepackage{hyperref}       
\usepackage{url}            
\usepackage{booktabs}       
\usepackage{amsfonts}       
\usepackage{nicefrac}       
\usepackage{microtype}      
\usepackage{lipsum}

\title{Moments approach for the elephant random walk}

\author{{\Large Pradeep Vishwakarma}\\\\
  Statistics and Mathematics Unit\\
  Indian Statistical Institute, Kolkata, 700108, West Bengal, India \\
  \texttt{vishwakarmapr.rs@gmail.com\thanks{The author acknowledge the support of National Post Doctoral Fellowship, PDF/2025/000076, from Anusandhan National Research Foundation (ANRF), Govt. of India}} \\
}

\begin{document}

\maketitle

\begin{abstract}

We discuss the method of moments for the one-dimensional elephant random walk (ERW). We first derive a differential recurrence relation for the characteristic function of the ERW, which yields a corresponding system of recurrence relations for its moments. We then obtain asymptotic approximations for the moments in each of the three parameter regimes of the ERW. Finally, by establishing the convergence of the moments and verifying the corresponding moment-determinacy conditions, we identify the limiting distributions of the ERW in each regime.
\end{abstract}


\section{Introduction} The ERW is a stochastic model in which the transition mechanism depends on the entire past trajectory of the walker, thereby incorporating complete memory of its history. The one-dimensional ERW was introduced by Sch{\"u}tz and Trimper \cite{Schutz2004} to investigate how long-term memory effects in a random walk can lead to the anomalous diffusive behavior. Let $X_1$ be a random variable such that $\mathbb{P}\{X_1=1\}=q$ and $\mathbb{P}\{X_1=-1\}=1-q$ for $q\in(0,1]$. For $n\ge1$, at time $n+1$, the step variable $X_{n+1}$ is given as 
\begin{equation*}
	X_{n+1}\coloneqq \begin{cases}
		+X_{U_n}\ \text{with probability $p$},\\
		-X_{U_n}\ \text{with probability $1-p$},
	\end{cases}
\end{equation*}
where $p\in(0,1)$ is called the memory parameter of the ERW, and $U_n$ is uniform random variable over $\{1,\dots,n\}$. The ERW, denoted by $\{S_n\}_{n\ge0}$ is defined as
\begin{equation}\label{Erwdef}
	S_n\coloneqq X_1+\dots+X_n,\ \ \text{with $S_0=0$}.
\end{equation}
In the past few years, the one-dimensional ERW has been studied extensively; see, for example, \cite{Baur2016, Bercu2018, Coletti2017a, Coletti2017b, Da Silva2013, Schutz2004, Kursten2016}. Its multidimensional extension has also been investigated (see \cite{Bercu2019, Bertenghi2022, Qin2025}). The asymptotic behavior of the one-dimensional ERW is determined by the memory parameter $p$ relative to the critical value $p=3/4$. In the diffusive regime, $p<3/4$, and at the critical value, $p=3/4$, both a strong law of large numbers and a central limit theorem (CLT) have been established (see \cite{Baur2016, Coletti2017a, Guevara2019}). In contrast, in the super-diffusive regime, $p>3/4$, the appropriately normalized position, $S_n/n^{2p-1}$, converges almost surely to a non-Gaussian limiting random variable (see \cite{Baur2016, Bercu2018, Coletti2017a, Guerin2025, Guerin2026}). For more recent developments concerning the ERW and its extensions, see \cite{Bertenghi2022, Coletti2021, Dhillon2025, Fan2021, Roy2024, Roy2025}, and the references therein.

The martingale approach is one of the most popular techniques used to investigate the limiting behavior of the ERWs. In this method, an appropriate martingale is constructed from the ERW, and martingale limit theorems are then employed to derive its asymptotic behavior; see, for example, \cite{Bercu2018, Bercu2019, Coletti2017a}. An alternative approach, based on the connection between the ERW and P\'olya urn models, was developed in \cite{Baur2016} and has also been used to study the limiting behavior of the walk.

Although these approaches provide detailed information about the asymptotic behavior of the ERW, they give only limited information about the limiting distribution in the super-diffusive regime. In particular, they yield the first four moments of the corresponding limiting random variable but do not provide a characterization of the full limiting distribution. More recently, Gu\'erin et al. \cite{Guerin2026} developed a fixed-point approach to the super-diffusive regime. They derived a distributional fixed-point equation for the limiting random variable and used it to establish the existence of its density. Furthermore, they obtained a nonlinear recurrence relation for the moments of the limiting distribution and proved that the distribution is uniquely determined by these moments. Subsequently, in \cite{Guerin2025}, the authors exploited the characterization obtained in \cite{Guerin2026} to investigate the asymptotic behavior of the density of the super-diffusive limit at infinity.

\vskip5pt

In this article, we present a moment approach for the one-dimensional ERW. In particular, we establish the CLT by proving the convergence of the moments of the appropriately normalized ERW, and subsequently identifying the limiting distribution. This approach provides a unified treatment of the diffusive and critical regimes and, importantly, establishes the characterization of the weak limit in the super-diffusive regime via its moments. Such a characterization was not obtained through the earlier martingale or P\'olya urn approaches. The super-diffusive limiting distribution was instead recently identified via a fixed-point equation for the limiting random variable (see \cite{Guerin2026}). The main contribution of the paper is summarized as follows:

In Section \ref{sec2}, we first derive a new system of differential equations governing the characteristic functions of the sequence $\{S_n\}_{n\geq1}$ (see Proposition \ref{prop1}). This system is then used to establish a recurrence relation for the moments of the ERW. In Theorem \ref{prop2}, we use this recurrence to obtain precise asymptotic approximations for the moments of the ERW in all three regimes: diffusive, critical, and super-diffusive.

In Theorem \ref{difflim}, we establish the central limit theorem in the diffusive and critical regimes by combining the convergence of the normalized moments with the fact that the Gaussian distribution is uniquely determined by its moments. In Theorem \ref{thm:uniquness}, we address the super-diffusive regime. We show that the limiting moments in this regime are non-Gaussian, and provide an alternate proof that the limiting moments uniquely identify the corresponding limiting probability distribution.

In Section \ref{proofs}, we provide the proofs of the results established in Section \ref{sec2}. Finally, in Section \ref{sec:conclusion}, we provide some concluding remarks.

\section{Moments of the ERW}\label{sec2}
 Here, we derive an approximation of the moments of the ERW using its characteristic function. Let $\{S_n\}_{n\ge0}$ be the ERW with first step $X_1$, whose characteristic function is given by
\begin{align}\label{cfs1}
	\mathbb{E}[e^{iuX_1}]&=qe^{iu}+(1-q)e^{-iu}=\cos(u)+i(2q-1)\sin(u),\ u\in\mathbb{R}.
\end{align}

Next, we derive a system of differential equations governing the characteristic functions of $\{S_n\}_{n\ge1}$.
\begin{proposition}\label{prop1}
	Let $\{S_n\}_{n\ge1}$ be the ERW as defined in (\ref{Erwdef}).\vskip5pt
	\noindent (i) For $n\ge1$, let $\phi_n(u)\coloneqq\mathbb{E}[e^{iuS_n}]$, $u\in\mathbb{R}$ be the characteristic function of $S_n$. It solves the following system of differential equations:
	\begin{equation}\label{re11}
		\phi_{n+1}(u)=\cos(u)\phi_n(u)+\frac{\alpha\sin(u)}{n}\frac{\mathrm{d}}{\mathrm{d}u}\phi_n(u),
	\end{equation}
	with $\phi_n(0)=1$ for each $n\ge1$, and where $\alpha\coloneqq2p-1\in[-1,1]$. \vskip5pt
	\noindent (ii) Let $\mu_{n}^{(r)}\coloneqq\mathbb{E}[S_n^r]$, $r\ge1$ be the $r$th moment of $S_n$. Then, 
	\begin{equation}\label{momentrel}
		\mu_{n+1}^{(r)}=\mathbb{I}\{r=\text{even}\}+\Big(1+\frac{\alpha}{n}r\Big)\mu_n^{(r)}+\sum_{\substack{1\leq l\leq r-1,\\ \text{$l$ is even}}}\bigg[\binom{r}{l}+\frac{\alpha}{n}\binom{r}{l+1}\bigg]\mu_n^{(r-l)},\ r\ge1,\ n\ge1,
	\end{equation}
	where $\mathbb{I}$ denotes the identity. Also, the empty sum is taken to be $0$.
\end{proposition}
\begin{remark}
From (\ref{momentrel}), for $n\ge1$, we get
\begin{align*}
	&\mathbb{E}[S_{n+1}]=\Big(1+\frac{\alpha}{n}\Big)\mathbb{E}[S_n],\\   &\mathbb{E}[S_{n+1}^2]=1+\Big(1+\frac{2\alpha}{n}\Big)\mathbb{E}[S_n^2],\\
    &\mathbb{E}[S_{n+1}^3]=\Big(1+\frac{3\alpha}{n}\Big)\mathbb{E}[S_n^3]+\Big(3+\frac{\alpha}{n}\Big)\mathbb{E}[S_n],\\
    &\mathbb{E}[S_{n+1}^4]=1+\Big(1+\frac{4\alpha}{n}\Big)\mathbb{E}[S_n^4]+\Big(6+\frac{4\alpha}{n}\Big)\mathbb{E}[S_n^2],
\end{align*}	
where the first two equations agree with those obtained in \cite{Schutz2004}.
\end{remark}

We now derive an approximation of the moments of ERW. 
\begin{theorem}\label{prop2} (i)
	Suppose $\alpha<\frac{1}{2}$. As $n\rightarrow\infty$, for $m\ge1$, we have
	\begin{equation}\label{momap1}
		\mu_n^{(r)}\sim\begin{cases}
			C_{2m}(\alpha)n^m,\ r=2m,\\
			O(n^{m-1+\alpha}),\ r=2m-1,
		\end{cases}
	\end{equation}
	where $C_0(\alpha)=1$ and
	\begin{equation}\label{coeffs}
		C_{2m}(\alpha)=\frac{\binom{2m}{2}}{m(1-2\alpha)}C_{2m-2}(\alpha), 
	\end{equation}
	\noindent (ii) If $\alpha=\frac{1}{2}$, then
	\begin{equation}
		\mu_n^{(r)}\sim\begin{cases}
			C_{2m}(n\log(n))^m,\ r=2m,\\
			O(n^{m-\frac{1}{2}}(\log(n))^{m-1}),\ r=2m-1,
		\end{cases}
	\end{equation}
	where $C_0=1$ and
	\begin{equation}\label{coeffs1}
	C_{2m}=\frac{\binom{2m}{2}}{m}C_{2m-2}, m\ge1.
	\end{equation}
    \noindent (iii) If $\alpha>\frac{1}{2}$, then
	\begin{equation}\label{moment3:appr}
		\mu_{n}^{(r)}\sim C_{n,r}(\alpha)n^{r\alpha},\ r\ge1,
	\end{equation}
	where $\mu_1^{(r)}=\mathbb{E}[X_1^r]$, $r\ge1$, and $\{C_{n,r}\}_{n\ge1}$  is a convergent sequence for each $r\ge1$, given as follows:
	\begin{align}
		C_{n,r}(\alpha)&=\frac{\mu_{1}^{(r)}}{\Gamma(1+r\alpha)}+\sum_{k=1}^{n-1}\frac{\Gamma(k+1)}{\Gamma(k+1+r\alpha)}\nonumber\\
        &\hspace{3cm} \cdot\Big[\mathbb{I}\{r=\text{even}\}+\sum_{\substack{1\leq l\leq r-1,\\ \text{$l$ is even}}}\bigg[\binom{r}{l}+\frac{\alpha}{k}\binom{r}{l+1}\bigg]C_{k,r-l}k^{(r-l)\alpha}\Big].\label{supdiff:mom}
	\end{align}
    \end{theorem}

\begin{remark}\label{remmom}
From (\ref{coeffs}), we have
	\begin{equation*}
		C_{2m}(\alpha)=\prod_{j=1}^{m}\frac{\binom{2j}{2}}{j(1-2\alpha)}=\frac{1}{(1-2\alpha)^m}\prod_{j=1}^{m}(2j-1)=\frac{(2m-1)!!}{(1-2\alpha)^m},\ m\ge1,\ \alpha<\frac{1}{2},
	\end{equation*}
	where $(2m-1)!!=(2m-1).(2m-3)\dots3.1$. 	
    Thus, $C_{2m}(\alpha)$ coincides with the $(2m)$th moments of a Gaussian rv with mean zero and variance $\frac{1}{1-2\alpha}$,  $\alpha<\frac{1}{2}$. Similarly, $C_{2m}$ as defined in (\ref{coeffs1}), coincides with the $(2m)$th moment of a standard Gaussian rv.
\end{remark}

\noindent{\bf CLT for the ERW via moments.} We now prove the central limit theorem for the ERW using the convergence of its moments. Previously, these results were established using the martingale method (see \cite{Bercu2018, Coletti2017a}).

First, we consider the diffusive and critical regimes. It is well known that the Gaussian distribution is uniquely determined by its moments. Therefore, in view of Remark \ref{remmom}, and using Theorem \ref{prop2} (i) and (ii), we immediately get the following result:

\begin{theorem}\label{difflim}
    \noindent (i) (Diffusive regime) If $\alpha<\frac{1}{2}$, then
    \begin{equation*}
        \frac{S_n-\mathbb{E}[S_n]}{\sqrt{n}}\overset{d}{\longrightarrow}\mathcal{N}(0,\frac{1}{1-2\alpha}),
    \end{equation*}
    where $\overset{d}{\longrightarrow}$ denotes the convergence in distribution.\vskip5pt
    \noindent (ii) (Critical regime) For $\alpha=\frac{1}{2}$,
    \begin{equation*}
        \frac{S_n-\mathbb{E}[S_n]}{\sqrt{n\log(n)}}\overset{d}{\longrightarrow}\mathcal{N}(0,1).
    \end{equation*}
\end{theorem}

Using a fixed point equation for the super-diffusive limit of the ERW, Gu\'erin et al. \cite{Guerin2026} recently established that its distribution is uniquely determined by the associated moments. We now identify the super-diffusive limit of the ERW using the moments convergence.

\begin{theorem}\label{thm:uniquness}
(i) (Supper-diffusive regime) For $\alpha>\frac{1}{2}$,
    \begin{equation*}
        M_r\coloneqq\lim_{n\rightarrow\infty}\mathbb{E}[(\frac{S_n}{n^{\alpha}})^r]=\lim_{n\rightarrow\infty}C_{n,r},\ r\ge1,
    \end{equation*}
    where $C_{n,r}$ is as defined in (\ref{supdiff:mom})\vskip3pt
\noindent (ii) The moments $\{M_r\}_{r\ge1}$ uniquely determine its probability distribution. In particular, if $L_\alpha$ is a random variable such that $M_r=\mathbb{E}[L_\alpha^r]$, then 
    \begin{equation*}
        \frac{S_n}{n^\alpha}\overset{d}{\longrightarrow} L_\alpha.
    \end{equation*}
\end{theorem}
\begin{remark}
    From (\ref{supdiff:mom}), for $\alpha>\frac{1}{2}$, we have
    \begin{equation*}
        \mathbb{E}[L_\alpha]=\lim_{n\rightarrow\infty}C_{n,1}=\frac{\mu_1^{(1)}}{\Gamma(1+\alpha)}=\frac{2q-1}{\Gamma(1+\alpha)},
    \end{equation*}
    and
    \begin{align*}
        \mathbb{E}[L_\alpha^2]=\lim_{n\rightarrow\infty}C_{n,2}&=\frac{\mu_{1}^{(2)}}{\Gamma(1+2\alpha)}+\sum_{k=1}^{n-1}\frac{\Gamma(k+1)}{\Gamma(k+1+2\alpha)}\\
        &=\frac{1}{\Gamma(1+2\alpha)}+\frac{1}{\Gamma(2\alpha)}\int_{0}^{1}t(1-t)^{2\alpha-2}\,\mathrm{d}t\\
        &=\frac{1}{\Gamma(1+2\alpha)}+\frac{\Gamma(2)\Gamma(2\alpha-1)}{\Gamma(2\alpha)\Gamma(1+2\alpha)}\\
        &=\frac{1}{(2\alpha-1)\Gamma(2\alpha)},
    \end{align*}
    which agree with the moments obtained in \cite{Bercu2018}.
\end{remark}

\section{Proofs of our results}\label{proofs}
\noindent{\bf Proof of Proposition \ref{prop1}.} We note the following fact that will be crucial in our proof:

Let $X$ be a rv taking only two values $\{-1,+1\}$ with positive probabilities. Then, for any real constant $c$, the following equalities hold almost surely: 
\begin{equation}\label{sincos:eql}
	\cos(cX)=\cos(c)\ \ \text{and}\ \ \sin(cX)=X\sin(c),
\end{equation}
where we have used that $\cos(-c)=\cos(c)$ and $\sin(-c)=-\sin(c)$.
\begin{proof}[{\bf Proof of Proposition \ref{prop1} (i)}] 
	Let $\mathcal{F}_n=\sigma(X_1,\dots,X_n)$, $n\ge1$, the sigma field generated by $(X_1,\dots,X_n)$. Then, 
	\begin{equation}\label{pf11}
		\phi_{n+1}(u)=\mathbb{E}[e^{iuS_n}\mathbb{E}[e^{iuX_{n+1}}|\mathcal{F}_n]],\ u\in\mathbb{R},
	\end{equation}
	where
	\begin{equation}\label{prop1:pf1}
		\mathbb{E}[e^{iuX_{n+1}}|\mathcal{F}_n]=\frac{1}{n}\Big[p\sum_{k=1}^{n}e^{iuX_k}+(1-p)\sum_{k=1}^{n}e^{-iuX_k}\Big],\ u\in\mathbb{R}.
	\end{equation}
	Note that $X_k=\pm 1$ for each $k=1,\dots,n$. Then, using (\ref{sincos:eql}), we get 
	\begin{align*}
		e^{iuX_k}&=\cos(uX_k)+i\sin(uX_k)=\cos(u)+iX_k\sin(u),\\
		e^{-iuX_k}&=\cos(uX_k)-i\sin(uX_k)=\cos(u)-iX_k\sin(u).
	\end{align*}
	Hence,
	\begin{equation*}
		pe^{iuX_k}+(1-p)e^{-iuX_k}=\cos(u)+i(2p-1)X_k\sin(u),\ k=1,\dots,n,
	\end{equation*}
	which on substituting in (\ref{prop1:pf1}) yields
	\begin{equation}\label{pf12}
		\mathbb{E}[e^{iuX_{n+1}}|\mathcal{F}_n]=\cos(u)+\frac{i\alpha\sin(u)}{n}S_n,
	\end{equation}
	where $\alpha=2p-1$. Now, on substituting (\ref{pf12}) in (\ref{pf11}), we get
	\begin{align}
		\phi_{n+1}(u)&=\cos(u)\mathbb{E}[e^{iuS_n}]+\frac{i\alpha\sin(u)}{n}\mathbb{E}[S_ne^{iuS_n}]\nonumber\\
		&=\cos(u)\phi_n(u)+\frac{\alpha\sin(u)}{n}\frac{\mathrm{d}}{\mathrm{d}u}\phi_n(u),\label{pf13}
	\end{align}
	where the interchange of expectation and derivative is justified because $\mathbb{E}[|S_ne^{iuS_n}|]<\infty$ for all $u\in\mathbb{R}$. This completes the proof of Proposition \ref{prop1} (i).
	\end{proof}	
    
	\begin{proof}[{\bf Proof of Proposition \ref{prop1} (ii)}] Note that  
	\begin{equation*}
		\mu_n^{(r)}=i^{3r}\frac{\mathrm{d}^r}{\mathrm{d}u^r}\phi_n(u)|_{u=0},\ r\ge1.
	\end{equation*}
	On taking the $r$th order derivative on both sides of (\ref{re11}) and using the Leibniz's theorem, we get
	\begin{equation}\label{meq1}
		\phi_{n+1}^{(r)}(u)=\sum_{l=0}^{r}\binom{r}{l}\cos^{(l)}(u)\phi_n^{(r-l)}(u)+\frac{\alpha}{n}\sum_{l=0}^{r}\binom{r}{l}\sin^{(l)}(u)\phi_n^{(r-l+1)}(u),
	\end{equation}
	where $f^{(l)}$ denotes the $l$th order derivative of the function $f$. On substituting $u=0$, and multiplying $i^{3r}$ on both sides of (\ref{meq1}) yields
	\begin{align*}
		\mu_{n+1}^{(r)}&=\sum_{\substack{0\leq l\leq r,\\ \text{$l$ is even}}}\binom{r}{l}i^li^{3l}\mu_n^{(r-l)}+\frac{\alpha}{n}\sum_{\substack{0\leq l\leq r,\\ \text{$l$ is odd}}}\binom{r}{l}i^{l-1}i^{3(l-1)}\mu_n^{(r-l+1)}\nonumber\\
		&=\mu_n^{(r)}+\sum_{\substack{1\leq l\leq r,\\ \text{$l$ is even}}}\binom{r}{l}\mu_n^{(r-l)}+\frac{\alpha}{n}\sum_{\substack{1\leq l\leq r,\\ \text{$l$ is odd}}}\binom{r}{l}\mu_n^{(r-l+1)}\nonumber\\
		&=\mu_n^{(r)}+\sum_{\substack{1\leq l\leq r,\\ \text{$l$ is even}}}\binom{r}{l}\mu_n^{(r-l)}+\frac{\alpha}{n}\sum_{\substack{0\leq l\leq r-1,\\ \text{$l$ is even}}}\binom{r}{l+1}\mu_n^{(r-l)}\nonumber\\
		&=\mathbb{I}\{r=\text{even}\}+\Big(1+\frac{\alpha}{n}r\Big)\mu_n^{(r)}+\sum_{\substack{1\leq l\leq r-1,\\ \text{$l$ is even}}}\bigg[\binom{r}{l}+\frac{\alpha}{n}\binom{r}{l+1}\bigg]\mu_n^{(r-l)},\ r\ge1,\ n\ge1,
	\end{align*}
	where we have used $\cos^{(l)}(u)|_{u=0}=i^{l}$, $\sin^{(l)}(u)|_{u=0}=i^{l-1}$, and $\mu_{n}^{(0)}=1$ for each $n\ge1$. This completes the proof of Proposition \ref{prop1} (ii).
\end{proof}

\noindent{\bf Proof of Theorem \ref{prop2}.}  
Before proving Theorem \ref{prop2}, we provide explicit calculations of moments for the ERW. The following approximation results will be needed:

\noindent(i) For $m\ge1$,
\begin{equation}\label{Knm:appr}
	K_{n,m}=\prod_{j=1}^{n-1}\Big(1+\frac{\alpha}{j}m\Big)=\frac{\Gamma(n+m\alpha)}{\Gamma(n)\Gamma(1+m\alpha)}\sim\frac{n^{m\alpha}}{\Gamma(1+m\alpha)},\ n>>1,
\end{equation}
where we have used
\begin{equation}\label{gammaration:appr}
	\frac{\Gamma(n+m\alpha)}{\Gamma(n)}=n^{m\alpha}(1+O(n^{-1})),\ m\ge1.
\end{equation}
\noindent (ii) For large $n$, we have
\begin{equation}\label{recipsum:appr}
	\sum_{j=1}^{n-1}j^c\sim\begin{cases}
		\frac{n^{1+c}}{1+c},\ c>-1,\\
		\log(n),\ c=-1,\\
		\text{convergent},\ c<-1.
	\end{cases}
\end{equation}

\noindent (iii) Let $\{a_n\}_{n\ge1}$ and $\{b_n\}_{n\ge1}$ be positive real sequences such that $a_n\sim b_n$, and $\sum_{j=1}^{n}b_j\rightarrow\infty$ as $n\rightarrow\infty$. Then, for large $n$,
\begin{equation}\label{eql:appr}
	\sum_{j=1}^{n}a_j\sim\sum_{j=1}^{n}b_j.
\end{equation}

For $r=1$, from (\ref{momentrel}), we have
\begin{equation*}
	\mu_{n+1}^{(1)}=\Big(1+\frac{\alpha}{n}\Big)\mu_n^{(1)}.
\end{equation*}
Using this recursion and $\mu_1^{(1)}=\mathbb{E}[X_1]=2q-1$, we get
\begin{equation}\label{ERWmom1:appr}
	\mu_{n}^{(1)}=\mu_{1}^{(1)}\prod_{k=1}^{n-1}\Big(1+\frac{\alpha}{k}\Big)=(1-2q)K_{n,1}\sim\frac{(2q-1)n^\alpha}{\Gamma(1+\alpha)},
\end{equation}
where we have used (\ref{Knm:appr}) to get the last step.

For $r=2$, from (\ref{momentrel}), we have
\begin{equation*}
	\mu_{n+1}^{(2)}=1+\Big(1+\frac{2\alpha}{n}\Big)\mu_{n}^{(2)}.
\end{equation*}
Then,
\begin{equation}\label{smom:appr}
	\mu_{n}^{(2)}=K_{n,2}\Big(\mu_1^{(2)}+\sum_{j=1}^{n-1}\frac{1}{K_{j+1,2}}\Big)\sim\frac{n^{2\alpha}}{\Gamma(2\alpha+1)}\Big(\mu_{1}^{(2)}+\sum_{j=1}^{n-1}\frac{\Gamma(2\alpha+1)}{j^{2\alpha}}\Big),
\end{equation}
where we have used (\ref{Knm:appr}) and (\ref{eql:appr}). Now, on using (\ref{recipsum:appr}) in (\ref{smom:appr}), we get
\begin{align*}
	\mu_{n}^{(2)}&\sim\begin{cases}
		\frac{n^{2\alpha}}{\Gamma(2\alpha+1)}(\mu_{1}^{(2)}+\frac{\Gamma(2\alpha+1)n^{1-2\alpha}}{1-2\alpha}),\ \alpha<\frac{1}{2},\\
		\frac{n^{2\alpha}}{\Gamma(2\alpha+1)}(\mu_{1}^{(2)}+\Gamma(2\alpha+1)\log(n)),\ \alpha=\frac{1}{2}
	\end{cases}\\
	&\sim\begin{cases}
		\frac{n}{1-2\alpha},\ \alpha<\frac{1}{2},\\
		n\log(n),\ \alpha=\frac{1}{2}.
	\end{cases}
\end{align*}

To prove Theorem \ref{prop2}, we use the method of strong induction on the values of $r$.
\begin{proof}[{\bf Proof of Theorem \ref{prop2} (i)}] Suppose $\alpha<\frac{1}{2}$, and (\ref{momap1}) holds true for $r=1,\dots,k$, $k\ge2$. From (\ref{momentrel}), for $r=k+1$, we have
	\begin{equation*}
		\mu_{n+1}^{(k+1)}=\mathbb{I}\{k+1=\text{even}\}+\Big(1+\frac{\alpha}{n}(k+1)\Big)\mu_n^{(k+1)}+\sum_{\substack{1\leq l\leq k,\\ \text{$l$ is even}}}\bigg[\binom{k+1}{l}+\frac{\alpha}{n}\binom{k+1}{l+1}\bigg]\mu_n^{(k+1-l)}.
	\end{equation*}
	\textbf{Case I} If $k$ is even, say $k=2m-2$ for some $m\ge2$, then $k+1=2m-1$ and $2m-1-l$ are odd for all even $1\leq l\leq k$. Hence, using the induction hypothesis and (\ref{momap1}), we get
	\begin{equation*}
		\mu_{n+1}^{(2m-1)}=\Big(1+\frac{\alpha}{n}(2m-1)\Big)\mu_n^{(2m-1)}+\sum_{\substack{1\leq l\leq 2m-2,\\ \text{$l$ is even}}}\bigg[\binom{2m-1}{l}+\frac{\alpha}{n}\binom{2m-1}{l+1}\bigg] O(n^{\frac{2m-2-l}{2}+\alpha}),\ n\ge1.
	\end{equation*}
	Note that
	\begin{equation*}
		g_n=\sum_{\substack{1\leq l\leq 2m-2,\\ \text{$l$ is even}}}\bigg[\binom{2m-1}{l}+\frac{\alpha}{n}\binom{2m-1}{l+1}\bigg] O(n^{\frac{2m-2-l}{2}+\alpha})\sim O(n^{m-2+\alpha}).
	\end{equation*}
	Then,
	\begin{align}
		\mu_{n+1}^{(2m-1)}&=K_{n,2m-1}\Big(\mu_{1}^{(2m-1)}+\sum_{j=1}^{n-1}\frac{g_j}{K_{j+1,2m-1}}\Big)\nonumber\\
		&\sim \frac{n^{(2m-1)\alpha}}{\Gamma(1+(2m-1)\alpha)}\Big(\mu_{1}^{(2m-1)}+\sum_{j=1}^{n-1}\Gamma(1+(2m-1)\alpha) O(j^{m-2+\alpha-(2m-1)\alpha})\Big)\label{pf221}\\
		&=n^{(2m-1)\alpha} O(n^{m-1+\alpha-(2m-1)\alpha})\nonumber\\
		&=O(n^{m-1+\alpha}),\nonumber
	\end{align}
	where we have used (\ref{Knm:appr}) in the second step, and the penultimate step follows from (\ref{recipsum:appr}) because $m-2+\alpha-(2m-1)\alpha>-1$ for $\alpha<\frac{1}{2}$.
	
	\textbf{Case II} Suppose $k=2m-1$ for some $m\ge2$. Then, $k+1=2m$ and $2m-l$ is even for all even $1\leq l\leq k$. Then, using induction hypothesis and (\ref{momap1}) in (\ref{momentrel}), we get
	\begin{align*}
		\mu_{n+1}^{(2m)}&=1+\Big(1+\frac{\alpha}{n}2m\Big)\mu_n^{(2m)}+\sum_{\substack{1\leq l\leq 2m-1,\\ \text{$l$ is even}}}\bigg[\binom{2m}{l}+\frac{\alpha}{n}\binom{2m}{l+1}\bigg]C_{(2m-l)}(\alpha)n^{\frac{2m-l}{2}}\\
		&\sim	1+\Big(1+\frac{\alpha}{n}2m\Big)\mu_n^{(2m)}+\bigg[\binom{2m}{2}+\frac{\alpha}{n}\binom{2m}{3}\bigg]C_{(2m-2)}(\alpha)n^{\frac{2m-2}{2}}\\
		&\sim \Big(1+\frac{\alpha}{n}2m\Big)\mu_n^{(2m)}+\binom{2m}{2}C_{(2m-2)}(\alpha)n^{m-1},\ n\ge1.
	\end{align*}
	Thus,
	\begin{align}
		\mu_{n}^{(2m)}&\sim K_{n,2m}\Big(\mu_{1}^{(2m)}+\binom{2m}{2}C_{(2m-2)}(\alpha)\sum_{j=1}^{n-1}\frac{j^{m-1}}{K_{j+1, 2m}}\Big)\nonumber\\
		&\sim \frac{n^{2m\alpha}}{\Gamma(1+2m\alpha)}\Big(\mu_{1}^{(2m)}+\binom{2m}{2}C_{(2m-2)}(\alpha)\sum_{j=1}^{n-1}\Gamma(1+2m\alpha)j^{(m-1)-2m\alpha}\Big)\label{pf222}\\
		&\sim \binom{2m}{2}C_{(2m-2)}(\alpha)n^{2m\alpha}\frac{n^{m-2m\alpha}}{m-2m\alpha}\nonumber\\
		&=\binom{2m}{2}C_{(2m-2)}(\alpha)\frac{n^m}{m(1-2\alpha)},\nonumber
	\end{align}
	where we have used (\ref{Knm:appr}) to get the second step, and the penultimate step follows from (\ref{recipsum:appr}). This completes the proof of Theorem \ref{prop2} (i),  using the method of strong induction.
	\end{proof}
	\begin{proof}[{\bf Proof of Theorem \ref{prop2} (ii)}]
	For $\alpha=\frac{1}{2}$, on using (\ref{recipsum:appr}) in (\ref{pf221}) and (\ref{pf222}), the proof of Theorem \ref{prop2} (ii) follows similarly to that of Part (i).
\end{proof}

\begin{proof}[{\bf Proof of Theorem \ref{prop2} (iii)}]
	Suppose $\alpha>\frac{1}{2}$. From (\ref{ERWmom1:appr}), the result holds true for $r=1$. For $r=2$, from (\ref{smom:appr}), we have
	\begin{align*}
		\mu_n^{(2)}&=K_{n,2}\mu_{1}^{(2)}+\sum_{k=1}^{n-1}\prod_{j=k+1}^{n-1}\Big(1+\frac{2\alpha}{j}\Big)\\
		&=K_{n,2}\mu_{1}^{(2)}+\sum_{k=1}^{n-1}\frac{\Gamma(n+2\alpha)\Gamma(k+1)}{\Gamma(n)\Gamma(k+1+2\alpha)}\\
        &\sim \frac{n^{2\alpha}}{\Gamma(1+2\alpha)}+n^{2\alpha}\sum_{k=1}^{n-1}\frac{\Gamma(k+1}{\Gamma(k+1+2\alpha)},
	\end{align*}
	where we have used (\ref{gammaration:appr}). Thus, (\ref{moment3:appr}) hold true for $r=2$. 
	
	Suppose (\ref{moment3:appr}) holds true for $r=1,2\dots,m-1$ for some $m\ge3$. For $r\ge3$, set
	\begin{align*}
		g_{n,r}&=\sum_{\substack{1\leq l\leq r-1,\\ \text{$l$ is even}}}\bigg[\binom{r}{l}+\frac{\alpha}{n}\binom{r}{l+1}\bigg]\mu_n^{(r-l)}\\
        &\sim \sum_{\substack{1\leq l\leq r-1,\\ \text{$l$ is even}}}\bigg[\binom{r}{l}+\frac{\alpha}{n}\binom{r}{l+1}\bigg]C_{n,r-l}n^{(r-l)\alpha}.
	\end{align*}
	From (\ref{momentrel}), we have
	\begin{equation*}
	\mu_{n+1}^{(m)}=	\mathbb{I}\{m=\text{even}\}+\Big(1+\frac{\alpha}{n}m\Big)\mu_n^{(m)}+g_{n,m},\ n\ge1.
	\end{equation*}
	Hence,
	\begin{align*}
		\mu_{n}^{(m)}&=\Big[K_{n,m}\mu_{1}^{(m)}+\sum_{k=1}^{n-1}\prod_{j=k+1}^{n-1}\Big(1+\frac{m\alpha}{j}\Big)[\mathbb{I}\{m=\text{even}\}+g_{k,m}]\Big]\\
		&=\Big[K_{n,m}\mu_{1}^{(m)}+\sum_{k=1}^{n-1}\frac{\Gamma(n+m\alpha)\Gamma(k+1)}{\Gamma(n)\Gamma(k+1+m\alpha)}[\mathbb{I}\{r=\text{even}\}+g_{k,m}]\Big]\\
		&\sim n^{m\alpha}\frac{\mu_{1}^{(m)}}{\Gamma(1+m\alpha)}+n^{m\alpha}\sum_{k=1}^{n-1}\frac{\Gamma(k+1)}{\Gamma(k+1+m\alpha)}\\
        &\ \ \cdot\Big[\mathbb{I}\{m=\text{even}\}+\sum_{\substack{1\leq l\leq m-1,\\ \text{$l$ is even}}}\bigg[\binom{m}{l}+\frac{\alpha}{k}\binom{m}{l+1}\bigg]C_{k,m-l}k^{(m-l)\alpha}\Big]\\
        &=C_{n,m}n^{m\alpha},
	\end{align*}
    where
    \begin{align*}
        C_{n,m}&=\frac{\mu_{1}^{(m)}}{\Gamma(1+m\alpha)}+\sum_{k=1}^{n-1}\frac{\Gamma(k+1)}{\Gamma(k+1+m\alpha)}\\
        &\ \ \cdot\Big[\mathbb{I}\{m=\text{even}\}+\sum_{\substack{1\leq l\leq m-1,\\ \text{$l$ is even}}}\bigg[\binom{m}{l}+\frac{\alpha}{k}\binom{m}{l+1}\bigg]C_{k,m-l}k^{(m-l)\alpha}\Big].
    \end{align*}
    In view of the induction hypothesis for each $1\leq r\leq m-1$, there exists a constant $K_r>0$ such that $|C_{k,r}|\leq K_r$. Then,
    \begin{align*}
        \Big|\sum_{k=1}^{n-1}&\frac{\Gamma(k+1)}{\Gamma(k+1+m\alpha)}\sum_{\substack{1\leq l\leq m-1,\\ \text{$l$ is even}}}\bigg[\binom{m}{l}+\frac{\alpha}{k}\binom{m}{l+1}\bigg]C_{k,m-l}k^{(m-l)\alpha}\Big|\\
        &\leq \sum_{\substack{1\leq l\leq m-1,\\ \text{$l$ is even}}}\Big[\binom{m}{l}\sum_{k=1}^{n-1}\frac{\Gamma(k+1)k^{(m-l)\alpha}}{\Gamma(k+1+m\alpha)}+K_{m-l}\binom{m}{l+1}\sum_{k=1}^{n-1}\frac{\Gamma(k+1)k^{(m-l)\alpha-1}}{\Gamma(k+1+m\alpha)}\Big]\\
        &\leq K_m'<\infty,
    \end{align*}
    where $K_m'$ is a finite constant, and to get the last step, we have used the following fact: The series 
    \begin{equation*}
        \sum_{k=1}^{\infty}\frac{\Gamma(k+1)k^a}{\Gamma(k+1+b)}
    \end{equation*}
    is convergent for all $a<b-1$. This completes the proof of Theorem \ref{prop2} (iii).
\end{proof}

\vskip5pt
\begin{proof}[{\bf Proof of Theorem \ref{thm:uniquness}}] The proof of Part (i) directly follows from Theorem \ref{prop2} (iii).

To prove the second part, we check the Carleman's condition (see \cite{Rohatgi2001}, p. 92). From Theorem \ref{prop2} (iii), we have that $\{C_{n,r}\}$ is a convergent sequence for each $r\ge1$. Set
    \begin{equation*}
        A_r\coloneqq\sup_{n}|C_{n,r}|,\ r\ge1.
    \end{equation*}
    First, we have the following claim:
    
    \textbf{Claim.} For an appropriate constant $K_\alpha$, $1/2<\alpha\leq1$ and for each $r\ge1$, we have $A_r\leq K^r_\alpha r!$.

    Let $M_r=\lim_{n\rightarrow\infty}C_{n,r}$. Then, $M_r\leq A_r$, and 
    \begin{equation*}
        M_{2r}\leq K_\alpha^{2r}(2r)!,\ r\ge1.
    \end{equation*}
    Consequently,
    \begin{equation*}
        [M_{2r}]^{-\frac{1}{2r}}\geq K_\alpha[(2r)!]^{-\frac{1}{2r}},\ r\ge1.
    \end{equation*}
    Using the Stirling approximation $(2r)!\sim\sqrt{4\pi r}(2r/e)^{2r}$, we get $[(2r)!]^{\frac{1}{2r}}\sim (4\pi r)^{1/4r}2r/e\sim 2r/e$. Thus,
    \begin{equation*}
        \sum_{r=1}^{\infty}[M_{2r}]^{-\frac{1}{2r}}\geq \frac{K_\alpha e}{2}\sum_{r=1}^{\infty}\frac{1}{r}=\infty.
    \end{equation*}
    Therefore, the moments $\{M_r\}_{r\geq1}$ satisfy Carleman's condition and uniquely determine the associated distribution.

    Now, it remains to prove \textbf{Claim}. We use induction to prove it. Clearly, the Claim holds for $r=1$. For $r=2$, we have 
    \begin{equation}\label{Unqpf:1}
        C_{n,2}=\frac{1}{\Gamma(1+2\alpha)}+\sum_{k=1}^{n-1}\frac{\Gamma(k+1)}{\Gamma(k+1+2\alpha)}.
    \end{equation}
    On using 
    \begin{equation}\label{zetasum:comp}
        \frac{\Gamma(k+1)}{\Gamma(k+1+c)}\leq k^{-c}\ \ \text{for $c>0$},
    \end{equation}
    we get
    \begin{equation}\label{Unqpf:2}
        \sum_{k=1}^{n-1}\frac{\Gamma(k+1)}{\Gamma(k+1+2\alpha)}\leq \sum_{k=1}^{\infty}k^{-2\alpha}=\zeta(2\alpha),
    \end{equation}
    where $\zeta(c)$ is the Riemann zeta function.
    On taking $K_\alpha>\zeta(2\alpha)/2+1/2\Gamma(1+2\alpha)$ and substituting (\ref{Unqpf:2}) in (\ref{Unqpf:1}), we get
    \begin{equation*}
        C_{n,2}<2K_\alpha.
    \end{equation*}
    Suppose the Claim holds for $r=1,\dots,m-1$, $m\ge2$. From (\ref{supdiff:mom}), we have
    \begin{align*}
        |C_{n,m}(\alpha)|&\leq\Big|\frac{\mu_{1}^{(m)}}{\Gamma(1+m\alpha)}\Big|+\sum_{k=1}^{n-1}\frac{\Gamma(k+1)}{\Gamma(k+1+m\alpha)}\nonumber\\
        &\hspace{2cm} \cdot\Big[\mathbb{I}\{m=\text{even}\}+\sum_{\substack{1\leq l\leq m-1,\\ \text{$l$ is even}}}\bigg[\binom{m}{l}+\frac{\alpha}{k}\binom{m}{l+1}\bigg]|C_{k,m-l}|k^{(m-l)\alpha}\Big]\\
        &\leq 1+\sum_{k=1}^{n-1}k^{-m\alpha}\Big[1+\sum_{\substack{1\leq l\leq m-1,\\ \text{$l$ is even}}}A_{m-l}\bigg[\binom{m}{l}+\frac{\alpha}{k}\binom{m}{l+1}\bigg]k^{(m-l)\alpha}\Big]\\
        &=1+\zeta(m\alpha)+\sum_{\substack{1\leq l\leq m-1,\\ \text{$l$ is even}}}K^{m-l}_\alpha(m-l)!\bigg[\zeta(l\alpha)\binom{m}{l}+\alpha\zeta(l\alpha+1)\binom{m}{l+1}\bigg],
    \end{align*}
    where we have used (\ref{zetasum:comp}) to get the second step, and the last step follows from the induction hypothesis. For $l\ge2$ and $1/2<\alpha\leq1$, we have $l\alpha>2\alpha>1$. As the Riemann zeta function is strictly decreasing on $(1,\infty)$, we have $\alpha\zeta(l\alpha+1)<\zeta(l\alpha)<\zeta(2\alpha)$, and
    \begin{equation*}
        |C_{n,m}(\alpha)|\leq 1+\zeta(2\alpha)+\zeta(2\alpha)K^m_\alpha m!\sum_{\substack{1\leq l\leq m-1,\\ \text{$l$ is even}}}K_\alpha^{-l}\Big[\frac{1}{l!}+\frac{(m-l)}{(l+1)!}\Big].
    \end{equation*}
    Now, we choose $K_\alpha$ sufficiently large such that $1+\zeta(2\alpha)\leq K_\alpha^mm!/2$ and
    \begin{equation*}
        \zeta(2\alpha)\sum_{\substack{1\leq l\leq m-1,\\ \text{$l$ is even}}}K_\alpha^{-l}\Big[\frac{1}{l!}+\frac{(m-l)}{(l+1)!}\Big]\leq \frac{1}{2}.
    \end{equation*}
    Then, $|C_{n,m}(\alpha)|\leq K_\alpha^mm!$. Thus, the proof of \textbf{Claim} is complete due to the method of induction. Note that all the above bounds are valid because $\zeta(c)\leq 1+\frac{1}{c-1}$ for $c>1$. Indeed, we can explicitly compute $K_\alpha$ by using the following bound:
    \begin{equation*}
        \sum_{\substack{l\ge2,\\ \text{$l$ is even}}}\frac{K_\alpha^{-l}}{l!}<K_\alpha^{-2}\sum_{l\ge1}\frac{1}{l!}=K_\alpha^{-2}e.
    \end{equation*}
    This completes the proof of Theorem \ref{thm:uniquness} (ii).
\end{proof}

\section{Conclusion and further discussion}\label{sec:conclusion}
We employed the method of moments to establish the CLT for the one-dimensional ERW. The key step in this approach was the derivation of a recursive system of differential equations for the characteristic function, from which we obtained a corresponding system of recurrence relations for the moments. These recurrences enabled us to determine the asymptotic behavior of the moments and, consequently, to establish the limiting distributions of the appropriately rescaled ERW in the three regimes, \textit{viz}, diffusive, critical, and super-diffusive.

A notable advantage of the moment approach is that it avoids the martingale and P\'olya urn constructions previously used in the analysis of the ERW, see \cite{Baur2016, Bercu2018, Coletti2017a}. Furthermore, the moment method provides additional information on the super-diffusive limit. In particular, the corresponding characterization is considerably more direct than the fixed-point approach utilized in \cite{Guerin2026}. We note that \cite{Guerin2026} also investigated the super-diffusive ERW in two and three dimensions. However, the moment method presented here is so far limited to the one-dimensional setting. More precisely, its central tool is the identity in (\ref{sincos:eql}), which allows us to derive the governing differential equation for the characteristic function. The absence of a direct analogue of this identity in higher dimensions prevents the present approach from being straightforwardly extended to multidimensional ERWs. Therefore, developing a suitable multidimensional counterpart to the present moment-based framework remains an interesting direction for investigation.

We now discuss the applicability of the moments method to the center of mass of the one-dimensional ERW. 

\noindent {\bf Center of mass of the ERW.} Let $\{S_n\}_{n\ge1}$ be the one-dimensional ERW. Its center of mass is defined as follows:
\begin{equation}\label{CoM}
    G_n\coloneqq\frac{1}{n}\sum_{k=1}^{n}S_k.
\end{equation}
It was introduced and studied (for multi-dimensional ERW) in \cite{Bercu2021}, where the authors derive various limit theorems for the process $\{G_n\}_{n\ge1}$ using a martingale approach.

We now consider the center of mass process $\{G_n\}_{n\ge1}$ of the one-dimensional ERW, defined in (\ref{CoM}). Note that for $n\ge1$,
\begin{equation}\label{comrep}
    G_{n+1}=\frac{n}{n+1}G_n+\frac{1}{n+1}S_{n+1}.
\end{equation}
Let $\Psi_n(u,v)\coloneqq\mathbb{E}[e^{\iota uG_n+\iota vS_n}]$, $u,v\in\mathbb{R}$ be the joint characteristic function of $(G_n, S_n)$. Then, $\Psi_1(u,v)=qe^{\iota (u+v)}+(1-q)e^{-\iota (u+v)}$.

Similar to Proposition \ref{prop1}, the following result gives a governing differential equation for $\Psi_n$.
\begin{proposition}
    For $n\ge1$ and $(u,v)\in\mathbb{R}^2$, $\{\Psi_n(u,v)\}$ solve the following system of differential equations:
    \begin{align}
        \Psi_{n+1}(u,v)&=\Big[\cos(\frac{u}{n+1}+v)+\frac{\alpha}{n}\sin(\frac{u}{n+1}+v)\frac{\mathrm{d}}{\mathrm{d}v}\Big]\Psi_{n}(\frac{nu}{n+1},\frac{u}{n+1}+v),\label{comcfequ1}
    \end{align}
    with $\Psi_n(0,0)=1$. 
\end{proposition}
\begin{proof}
    Using (\ref{comrep}), we get
    \begin{equation}\label{prop21pf11}
        \Psi_{n+1}(u,v)=\mathbb{E}[\mathbb{E}[e^{\iota uG_{n+1}+\iota vS_{n+1}}|\mathcal{F}_n]]=\mathbb{E}[e^{\iota\frac{n}{n+1}uG_n}\mathbb{E}[e^{\iota(\frac{u}{n+1}+v)S_{n+1}}|\mathcal{F}_n]].
    \end{equation}
    Here,
    \begin{align}
        \mathbb{E}[e^{\iota(\frac{u}{n+1}+v)S_{n+1}}|\mathcal{F}_n]&=e^{\iota(\frac{u}{n+1}+v)S_{n}} \mathbb{E}[e^{\iota(\frac{u}{n+1}+v)X_{n+1}}|\mathcal{F}_n]\nonumber\\
        &=e^{\iota(\frac{u}{n+1}+v)S_{n}} [\cos(\frac{u}{n+1}+v)+\frac{\iota\alpha}{n}\sin(\frac{u}{n+1}+v)S_n]\nonumber\\
        &=\cos(\frac{u}{n+1}+v)e^{\iota(\frac{u}{n+1}+v)S_{n}}+\frac{\alpha}{n}\sin(\frac{u}{n+1}+v)\frac{\mathrm{d}}{\mathrm{d}v}e^{\iota(\frac{u}{n+1}+v)S_{n}},\label{prop21pf12}
    \end{align}
    where we have used (\ref{pf12}) to obtain the second equality. On substituting (\ref{prop21pf12}) in (\ref{prop21pf11}), we get
    \begin{align*}
        \Psi_{n+1}(u,v)&=\cos(\frac{u}{n+1}+v)\mathbb{E}[e^{\iota \frac{n}{n+1}uG_{n}+\iota(\frac{u}{n+1}+v)S_n}]+\frac{\alpha}{n}\sin(\frac{u}{n+1}+v)\frac{\mathrm{d}}{\mathrm{d}v}\mathbb{E}[e^{\iota \frac{n}{n+1}uG_{n}+\iota(\frac{u}{n+1}+v)S_n}]\\
        &=\cos(\frac{u}{n+1}+v)\Psi_{n}(\frac{nu}{n+1},\frac{u}{n+1}+v)+\frac{\alpha}{n}\sin(\frac{u}{n+1}+v)\frac{\mathrm{d}}{\mathrm{d}v}\Psi_{n}(\frac{nu}{n+1},\frac{u}{n+1}+v),
    \end{align*}
    where the interchange of integral and derivative is valid due to $\mathbb{E}[\frac{\mathrm{d}}{\mathrm{d}v}e^{\iota \frac{n}{n+1}uG_{n}+\iota(\frac{u}{n+1}+v)S_n}]<\infty$. This completes the proof. 
\end{proof}
\begin{remark}
Note that for $u=0$, we get $\phi(v)=\Psi(0,v)$, and the system of equations (\ref{comcfequ1}) reduces to (\ref{re11}).
   \end{remark}

\bibliographystyle{plain}

\end{document}